\documentclass[11pt,reqno]{amsart}

\usepackage{fullpage}
\usepackage{setspace}
\usepackage{xcolor}
\usepackage[backref]{hyperref}  
\usepackage{hyperref}  
\usepackage{graphicx}
\usepackage{datetime}
\usepackage{verbatim}
\usepackage{srcltx}

\usepackage{latexsym}
\usepackage{amsfonts}
\usepackage{amsmath}

\def\EE{\mathbb E}
\def\cE{\mathcal E}

\def\cH{\mathcal H}

\def\cG{\mathcal G}

\def\cH{\mathcal H}

\def\cH{\mathcal{G}}

\def\({\left(}
\def\){\right)}  
\def\<{\langle}
\def\>{\rangle}
\let\:=\colon
\let\epsilon\varepsilon
\let\log\log

\newcommand{\hm}[1]{\leavevmode{\marginpar{\tiny%
$\hbox to 0mm{\hspace*{-0.5mm}$\leftarrow$\hss}%
\vcenter{\vrule depth 0.1mm height 0.1mm width \the\marginparwidth}%
\hbox to 0mm{\hss$\rightarrow$\hspace*{-0.5mm}}$\\\relax\raggedright #1}}}

\def\({\left(}
\def\){\right)}
\def\:{\colon}
\def\[{\left[}
\def\]{\right]}

\def\cH{\mathcal H}

\let\epsilon\varepsilon

\DeclareMathAlphabet{\Sss}{U}{bbmss}{m}{n}

\usepackage{amssymb}
\newcommand{\ignore}[1]{{}}



\newtheoremstyle{note}
  {4pt}
  {4pt}
  {\sl}
  {}
  {\bfseries}
  {.}
  {.5em}
  {}

\newtheoremstyle{introthms}
  {3pt}
  {3pt}
  {\normalfont}
  {}
  {\bfseries}
  {.}
  {.5em}
  {\thmnote{#3}}

\newtheoremstyle{cases}
  {2pt}
  {2pt}
  {\rm}
  {}
  {\bfseries}
  {.}
  {.3em}
  {}

\theoremstyle{plain}
\newtheorem{theorem}              {Theorem}       
\newtheorem{claim}      [theorem] {Claim}         
\newtheorem{lemma}      [theorem] {Lemma}         
\newtheorem{corollary}  [theorem] {Corollary}     

\theoremstyle{cases}

\theoremstyle{introthms}

\theoremstyle{note} 
\newtheorem{remark}    [theorem]   {Remark}       
\newtheorem{definition}[theorem]   {Definition}

\usepackage{lineno}
\newcommand*\patchAmsMathEnvironmentForLineno[1]{%
\expandafter\let\csname old#1\expandafter\endcsname\csname #1\endcsname
\expandafter\let\csname oldend#1\expandafter\endcsname\csname end#1\endcsname
\renewenvironment{#1}%
{\linenomath\csname old#1\endcsname}%
{\csname oldend#1\endcsname\endlinenomath}}%
\newcommand*\patchBothAmsMathEnvironmentsForLineno[1]{%
\patchAmsMathEnvironmentForLineno{#1}%
\patchAmsMathEnvironmentForLineno{#1*}}%
\AtBeginDocument{%
\patchBothAmsMathEnvironmentsForLineno{equation}%
\patchBothAmsMathEnvironmentsForLineno{align}%
\patchBothAmsMathEnvironmentsForLineno{flalign}%
\patchBothAmsMathEnvironmentsForLineno{alignat}%
\patchBothAmsMathEnvironmentsForLineno{gather}%
\patchBothAmsMathEnvironmentsForLineno{multline}%
}

\title[AKPSS for non-unif]{%
The independence number of non-uniform uncrowded hypergraphs and an anti-Ramsey type result }

\author[S.~J.~Lee]{Sang June Lee}
\address{Department of Mathematics\\ Duksung Women's University, Seoul 01369, South Korea}
\email{sanglee242@duksung.ac.kr, sjlee242@gmail.com}

\author[H.~Lefmann]{Hanno Lefmann}
\address{Fakult\"at f\"ur Informatik,
  TU Chemnitz, 
 D-09107 Chemnitz, Germany}
\email{lefmann@informatik.tu-chemnitz.de}

\thanks{}

\date{\today, \currenttime}

\begin{document}



\shortdate
\settimeformat{ampmtime}


\begin{abstract} We prove the following: Fix an integer $k\geq 2$, and let $T$ be a real number with $T\geq 1.5$. Let $\cH=(V,\cE_2\cup \cE_3\cup\dots\cup\cE_k)$ be a non-uniform hypergraph with the vertex set $V$ and the set $\cE_i$ of edges of size $i=2,\ldots , k$.
Suppose that $\cH$ has no $2$-cycles (regardless of sizes of edges),
 and  neither contains $3$-cycles nor $4$-cycles consisting of $2$-element edges. If the average degrees
$t_i^{i-1} := i  |\cE_i|/ |V|$ 
 satisfy that 
  $t_i^{i-1}  \leq 
T^{i-1}  (\ln T)^{\frac{k-i}{k-1}}$ for $i= 2, \dots , k$, then there exists a constant $C_k > 0$, depending only on $k$, such that
$\alpha(\cH)\geq C_k    \frac{|V|}{T}  (\ln T)^{\frac{1}{k-1}}$, where $\alpha(\cH)$ denotes the independence number of $\cH$. This extends results of   Ajtai, Koml\'os, Pintz, Spencer and Szemer\'edi~
[J. Comb. Theory Ser. A 32, 1982, 321--335] 
   and Duke, R\"odl and the second author~
 [Random Struct. Algorithms 6, 1995, 209--212] for uniform hypergraphs.

 As an application, we consider an anti-Ramsey type problem on non-uniform hypergraphs. Let $\cH=\cH(n;2,\ldots,\ell)$ be the hypergraph on the $n$-vertex set $V$ in which, for $s=2,\ldots,\ell$, each $s$-subset of $V$ is a hyperedge of $\cH$. Let $\Delta$ be an edge-coloring of $\cH$ satisfying the following:
(a) two hyperedges sharing a vertex have different colors; 
(b) two hyperedges with distinct size have different colors;
(c) a color used for a hyperedge of size $s$ appears at most $u_s$ times.
For such a coloring $\Delta$, let $f_{\Delta}(n;u_2,\ldots,u_{\ell})$ be the maximum size of a  subset $U$ of $V$ such that each hyperedge of $\cH[U]$ has a distinct color, and let
 $f(n;u_2,\ldots,u_{\ell}):=\min_{\Delta} f_{\Delta}(n;u_2,\ldots,u_{\ell}).$
We determine $f(n;u_2,\ldots,u_{\ell})$ up to a multiplicative logarithm factor, which is a non-uniform version of a result for edge-colorings of graphs by Babai~
[Graphs Comb. 1, 1985, 23--28], and for uniform hypergraphs by Alon, R\"odl and the second author~
[Coll.\ Math.\ Soc.\ J\'anos Bolyai, 60.\ Sets, Graphs
and Numbers, 1991, 9--22] and by R\"odl, Wysocka and  the second author~
 [J. Comb. Theory Ser. A 74, 1996, 209--248].

\end{abstract}

\maketitle

\section{Introduction}

Let $\cH = (V, \cE)$ be a hypergraph with its vertex set $V$ and its edge set $\cE$. Let $\cE_i\subset \cE$ be the set of all $i$-element edges in $\cH$.
 For a vertex $v \in V$, let $d_i(v)$
denote the number of $i$-element edges $E \in \cE_i$ containing $v$.
A hypergraph $\cH = (V, \cE)$  is called \emph{$k$-uniform} if 
$\cE=\cE_k$. A subset $V' \subseteq V$ is called \emph{independent} if for
 no edge $E \in \cE$ it is $E \subseteq V'$.
The \emph{independence number}
$\alpha (\cH)$ of $\cH$
 is the maximum size of an independent set of $\cH$. For a subset
$V^* \subset V$ of the vertex set, let $\cH [V^*]$ be the subhypergraph of $\cH$ induced on $V^*$.

The independence number has been well-studied for uniform hypergraphs,
however,  for  non-uniform hypergraphs it was not that
studied correspondingly. The goal of this paper is to extend known results on the independence number of a uniform hypergraph to a non-uniform hypergraph.

Tur\'an's theorems by Tur\'an~~\cite{turan} and Spencer~~\cite{Sp72} imply the following theorem about the independence number of a $k$-uniform hypergraph $\cH$. 

\begin{theorem}[Tur\'an~~\cite{turan} and Spencer~~\cite{Sp72}] \label{turan.2} Let $k\geq 2$ be an integer.
Let $\cH =(V, \cE_k)$ be a $k$-uniform hypergraph with $N$ vertices and 
average degree
$t^{k-1} := k  |\cE_k|/N$, where $t \geq 1$.
Then, 
\begin{eqnarray} \label{stern1}
\alpha (\cH) \geq \frac{k-1}{k}  \frac{N}{t} \, .
\end{eqnarray}
\end{theorem}

Later, a better lower bound on the independence number of a $k$-uniform hypergraph was obtained if the hypergraph does not contain cycles of small lengths. We introduce the definition of cycles of a given length.

\begin{definition} Let $\cH = (V, \cE)$ be a hypergraph.
A \emph{$j$-cycle} in $\cH$ is a family of pairwise
distinct edges $E_1, \dots , E_j \in \cE$ such that the following hold:
\begin{itemize}
\item $E_i 
\cap E_{i+1} \neq \emptyset$ for $i = 1, \dots , j-1$ and
$E_j \cap E_1 \neq \emptyset$.
\item There are pairwise distinct 
vertices $v_1, \dots, v_{j}$ such that $v_{i} \in E_{i} \cap E_{i+1}$ for 
$i = 1, \dots ,
j-1$ and $v_j \in E_j \cap E_1$.
\end{itemize}
\end{definition}

\begin{definition}
 A hypergraph $\cH$ 
is called  \emph{uncrowded} if it does not contain any $2$-, $3$-
or $4$-cycles.
  A hypergraph $\cH$ is called \emph{linear} if  it does not contain any $2$-cycles.
\end{definition}

Ajtai, Koml\'os, Pintz, Spencer and Szemer\'edi~~\cite{AKPSS82} obtained a lower bound on the independence number of a $k$-uniform uncrowded hypergraph as follows. Later, Bertram-Kretzberg and Lefmann~~\cite{BLsoda} and Fundia~~\cite{Fu96} provided  
a deterministic polynomial time algorithm.

\begin{theorem}[Ajtai, Koml\'os, Pintz, Spencer and Szemer\'edi~~\cite{AKPSS82}] \label{theo.uncrow}
Let $k \geq 2$ be a fixed integer. Let $t$ and $N$ satisfy $t > t_0(k)$ 
and $ N > N_0(k,t)$.
 Let $\cH = (V, \cE_k)$ be an uncrowded $k$-uniform 
hypergraph on $N$ vertices with average degree
$t^{k-1} := k  |\cE_k|/ N$. 
Then, there exists a constant $C_k > 0$ such that
\begin{eqnarray} \label{spe_szem}
\alpha (\cH) \geq
C_k    \frac{N}{t}  (\ln t)^{\frac{1}{k-1}}  \; .
\end{eqnarray}
\end{theorem}
We remark that in~~\cite{AKPSS82} the constant $C_k$
 is bounded from below by
$C_k \geq \frac{0.98}{e}    10^{-\frac{5}{k-1}}$.

Duke, Lefmann, and R\"odl in~~\cite{DLR95} weakened the assumption in Theorem~\ref{theo.uncrow} for $k\geq 3$ as follows:
\begin{theorem}[Duke, Lefmann, and R\"odl~~\cite{DLR95}] \label{theo.linear}
Let $k \geq 3$ be a fixed integer. Let $t$ and $N$ satisfy $t > t_0(k)$ 
and $ N > N_0(k,t)$.
 Let $\cH = (V, \cE_k)$ be a linear $k$-uniform 
hypergraph on $N$ vertices with  average degree
$t^{k-1} := k  |\cE_k|/ N$. 
Then, there exists a constant $C'_k > 0$ such that
\begin{eqnarray} \label{spe_szem1}
\alpha (\cH) \geq
C'_k    \frac{N}{t}  (\ln t)^{\frac{1}{k-1}}  \; .
\end{eqnarray}
\end{theorem}

In this paper we extend  Theorems 
\ref{theo.uncrow} and~\ref{theo.linear} to
non-uniform hypergraphs as follows. 

\begin{theorem} \label{theo.main-new}
Let $ k\geq 2$ be a fixed integer.  Let $T$ be a real number with $T\geq 1.5$ and $N$ be a positive integer.  Let $\cH = (V, \cE_2 \cup \dots \cup \cE_k)$
 be a linear
hypergraph on $N$ vertices such that  there are no $3$-cycles and $4$-cycles consisting of $2$-element edges.  Let the average degrees
$t_i^{i-1} := i  |\cE_i|/ N$ 
 satisfy that  for $i= 2, \dots , k$
  \begin{equation}\label{eq:main_condition}
  t_i^{i-1}  \leq  
T^{i-1}  (\ln T)^{\frac{k-i}{k-1}}.
\end{equation}
Then, there exists a constant $C_k > 0$ such that
\begin{eqnarray} \label{szem-new}
\alpha (\cH) \geq
C_k    \frac{N}{T}  (\ln T)^{\frac{1}{k-1}}  \; .
\end{eqnarray}
\end{theorem}

 \begin{remark}   For the range of $0<T<1.5$ in Theorem~\ref{theo.main-new},
a simple greedy algorithm gives that there exists a constant $C'_k>0$, depending only on $k$, such that
$\alpha(\cH)\geq C'_kN.$
\end{remark}

As an application of the main theorem (Theorem~\ref{theo.main-new}), we consider a anti-Ramsey type  problem of a non-uniform hypergraph.
Let $\cH=\cH(n;2,\ldots,\ell)$ be the hypergraph on the vertex set $V$, $|V|=n$, in which, for $s=2,\ldots,\ell$, each $s$-subset of $V$ is a hyperedge of $\cH$. Suppose that $ \Delta$ is an edge-coloring of $\cH$ satisfying the following conditions:
\begin{enumerate}
\item Two hyperedges sharing a vertex have different colors. In other words, each color class is a \emph{matching}.
\item Two hyperedges with distinct size have different colors.
\item The coloring is \emph{$(u_2,\ldots,u_{\ell})$-bounded}, that is, a color used for a hyperedge of size $s$ appears at most $u_s$ times.
\end{enumerate}

For such a coloring $ \Delta$, let $f_{\Delta}(n;u_2,\ldots,u_{\ell})$ be the maximum size of a vertex set $U$ in $V$ such that the subhypergraph $\cH[U]$ of $\cH$ induced on $U$ is \emph{totally multicolored}, which means that 
 each hyperedge of $\cH[U]$ has a distinct color. Let
 $$f(n;u_2,\ldots,u_{\ell}):=\min_{\Delta} f_{\Delta}(n;u_2,\ldots,u_{\ell}),$$
 where we minimize it over all edge-colorings $\Delta$   of $\cH$ satisfying (a)--(c).
 
  We will show the following  which is a non-uniform version of a result for edge-colorings of graphs in~~\cite{LRW96}.

\begin{theorem}\label{thm:application_a} $\phantom{ }$ 
\begin{enumerate}
\item Suppose that 
\begin{equation}\label{cond1_a}\max_{2\leq i\leq \ell }\[\(n^{i-1}u_i\)^{\frac{1}{2i-1}}\]=\(n^{s-1}u_s\)^{\frac{1}{2s-1}}.
\end{equation}
Then, there exist positive constants $c_1$ and $c_2$, depending only on $\ell$, such that for every sufficiently large $n$,
\begin{equation}\label{eq:result_a} c_1 \(\frac{n^s}{u_s}\)^{\frac{1}{2s-1}} \leq f(n;u_2,...,u_{\ell})\leq c_2 \(\frac{n^s}{u_s}\ln n\)^{\frac{1}{2s-1}}.
\end{equation}

\item Suppose that 
\begin{equation}\label{cond2_a}
	 \max_{2\leq i\leq \ell }\[\(\frac{n^{i-1}u_i} {\ln n}\)^{\frac{1}{2i-1}}\]=\(\frac{n^{s-1}u_s}{\ln n}\)^{\frac{1}{2s-1}}
\end{equation}
 and
\begin{equation}\label{cond3_a}
		u_s \geq n^{1/2 + \varepsilon} \hskip 2em \mbox{ for some absolute constant $\varepsilon > 0$}. 
\end{equation}
Then, there exist positive constants  $c_3=c_3(\ell, \epsilon)$ and $c_4=c_4(\ell)$ such that  for every sufficiently large $n$,
\begin{equation}\label{eq:result_b}c_3  \(\frac{n^s}{u_s}\ln n\)^{\frac{1}{2s-1}}  \leq f(n;u_2,...,u_{\ell}) \leq c_4 \(\frac{n^s}{u_s}\ln n\)^{\frac{1}{2s-1}}.
\end{equation}
\end{enumerate}
\end{theorem}

\begin{remark} 
Statements (a) and (b) in Theorem~\ref{thm:application_a}  have different assumptions. We remark that under asumption~\eqref{cond2_a}, our argument for the lower bound in~\eqref{eq:result_a} gives
	\begin{eqnarray*} \label{se}
c_3'  \(\frac{n^s}{u_s}\)^{\frac{1}{2s-1}}  \( \ln n \)^{\frac{ 1}{2s-1}-\frac{ 1}{2\ell-1}} \leq f(n;u_2,...,u_{\ell}),
\end{eqnarray*}
 which is less than the lower bound in~\eqref{eq:result_b},
while under assumption~\eqref{cond1_a},  our argument  for the lower bound in~\eqref{eq:result_b}
		gives 
\begin{eqnarray*}\label{se2}
c_3''  \(\frac{n^s}{u_s}  \)^{\frac{1}{2s-1}} 
\max\[1, \(\ln \frac{u_s^2}{n}\)^{\frac{1}{2\ell-1}}\] \leq f(n;u_2,...,u_{\ell}),
\end{eqnarray*}
  which is bigger than the lower bound in~\eqref{eq:result_a} if $u_s\gg \sqrt{n}$.
\end{remark}

 If we assume $u_2 = \ldots = u_{\ell} = n$, Theorem~\ref{thm:application_a} (b) immediately implies the following corollary:
\begin{corollary} There exist positive constants $c_3$ and $c_4$, depending only on $\ell$, such that for every sufficiently large $n$,
 \begin{eqnarray}  \label{eq:same_u}
c_3 \(n \ln n\)^{\frac{1}{3}}  &\leq& f(n;\underbrace{n,...,n}_{\ell-1}) \leq c_4 \( n\ln n\)^{\frac{1}{3}}.
 \end{eqnarray}
 \end{corollary} 
 
 For the special case of $\ell = 2$, i.e.,\ edge-coloings of graphs, the upper bound in~\eqref{eq:same_u} was shown  in 1985 by Babai~~\cite{Ba85}, and the lower bound was proved  in 1991
by Alon, R\"odl and the second author~~\cite{ALR91}.

The organization of this paper is as follows. Section~\ref{sec:turan} contains
 an extension of Theorem~\ref{turan.2} to non-uniform hypergraphs.
 In Section~\ref{sec:uncrowded}, we show an extension of
 Theorem~\ref{theo.uncrow} to non-uniform hypergraphs.
 Section~\ref{sec:linear with large} contains our proof of an extension of
 Theorem~\ref{theo.linear} to non-uniform hypergraphs. In Section~\ref{sec:linear}, we show Theorem~\ref{theo.main-new} in full.  Section~\ref{sec:proof_application} contains the proof of Theorem~\ref{thm:application_a}.


\section{Arbitrary non-uniform hypergraphs}\label{sec:turan}

The next theorem by Spencer~~\cite{Sp72} provides a lower bound on the
independence number $\alpha(\cH)$ of a non-uniform hypergraph $\cH$.

\begin{theorem}[Spencer~~\cite{Sp72}] \label{turan.2.non}
Let $\cH =(V,\cE_2 \cup \dots \cup \cE_k)$
 be a hypergraph with 
$N$ vertices  and
average degree 
$t_i^{i-1} := i  |\cE_i|/N$
for the $i$-element edges, where $i = 2, \dots , k$. 
Let $T := \mbox{ max } \{ t_i \; | \; 2 \leq i \leq k\} \geq 1/2$.
Then,  
$$
\alpha (\cH) \geq \frac{1}{4} \frac{N}{T}. 
$$
\end{theorem}

\begin{proof}
For convenience we provide the sketch of a proof.
Choose each vertex uniformly at random and independently with probability $p := 1/(2  T)\leq 1$. 
Let $V^*$ be the random set of chosen vertices and let
$\cE_i^*$, $i = 2, \dots , k$,
be the sets of $i$-element edges in the induced random subhypergraph 
$\cH[V^*]$. 
Then we have in expectation
\begin{eqnarray*}
\EE\[|V^*|-\sum_{i=2}^k |\cE_i^*|\] &=&
\EE[|V^*|]  - \sum_{i=2}^k \EE[|\cE_i^*|] 
= p  N  - \sum_{i=2}^k \frac{p^i  N  t_i^{i-1}}{i}  \\
&\geq &
p  N  - \sum_{i=2}^k \frac{p^i  N  T^{i-1}}{i} =
\frac{N}{2   T} -  \sum_{i=2}^k 
\frac{1}{i  2^i}  \frac{N}{ T} \geq \frac{N}{4  T}  . 
\end{eqnarray*}
Thus there exists a subset $V^* \subseteq V$
such that 
$$
|V^*|  - \sum_{i=2}^k |\cE_i^*| \geq 
\frac{N}{4  T}.
$$
By deleting one vertex from each edge $E \in \cE_i^*$,
$i=2, \dots , k$, we destroy all edges in $\cH[V^*]$, and hence we obtain an independent set
$V^{**}$ of $\cH$ with $|V^{**}| \geq  N/(4  T)$.
\end{proof}

\section{Uncrowded non-uniform hypergraphs}\label{sec:uncrowded}

In this section we show the following lemma which gives a lower bound on the independence number of an uncrowded non-uniform hypergraph.
\begin{lemma} \label{lem.uncrow.new}
Let $ k\geq 2$ be a fixed integer. Let $T$ and $N$ satisfy $T > T_0(k)$ 
and $ N > N_0(k, T)$.
  Let $\cH = (V, \cE_2 \cup \dots \cup \cE_k)$
 be an uncrowded 
hypergraph on $N$ vertices with average degrees $t_i^{i-1} := i  |\cE_i|/ N$ such that  for $i= 2, \dots , k$ $$t_i^{i-1}  \leq c_i  
T^{i-1}  (\ln T)^{\frac{k-i}{k-1}},
$$
where $c_i$ are constants satisfying $0 < c_i < \frac{1}{{16}k^2}  {k-1 \choose i-1}  
10^{-3 \frac{(k-i)}{k-1}} $.
Then, there exists a constant $C_k > 0$ such that\begin{eqnarray} \label{tu1}
 \alpha (\cH) \geq C_k  \frac{N}{T}  (\ln T)^{\frac{1}{k-1}}.
\end{eqnarray}
\end{lemma}

In order to prove Lemma~\ref{lem.uncrow.new}, we will use the following lemma which was used to show Theorem~\ref{theo.uncrow} in Ajtai, Koml\'os, Pintz, Spencer and 
Szemer\'edi~~\cite{AKPSS82}. 
Since the lemma in~~\cite{AKPSS82} has been written densely (with several typos), we provide a slightly modified statement, based on 
Lemma~6 in Fundia~~\cite{Fu96}.
\begin{lemma}[Ajtai, Koml\'os, Pintz, Spencer and 
Szemer\'edi~~\cite{AKPSS82}] \label{lemma1}
Let $k\geq 2$ be a fixed integer.
Let $T$, $N$, and $s$ be such that 
\begin{equation}\label{eq:s}
T\geq T_0(k), \hskip 0.5em N\geq N_0(k,T), \hskip 0.5em 0 \leq s \leq 0.01\ln T.
\end{equation}
Let $n$ and $t$ be integers satisfying
\begin{equation}\label{eq:n,t}
\frac{1}{2}\frac{N}{e^s}\leq n\leq \frac{N}{e^s}  \hskip 1em \mbox{and} \hskip 1em \frac{T}{e^s}\leq t\leq 2\frac{T}{e^s}. 
\end{equation}
Let $\cH = (V, \cE_2 \cup 
\dots \cup \cE_k)$ be an uncrowded hypergraph with $n$ vertices satisfying the following: 
for each vertex $v \in V$ and $2\leq i\leq k$, the numbers $d_i(v)$
of $i$-element edges $E \in\cE_i$ containing vertex $v$ fulfill
$$d_i(v) \leq {k-1 \choose i-1}  s^{\frac{k-i}{k-1}}  t^{i-1}.$$

Then, there exists an independent set $I \subseteq V$
and a vertex set $V^* \subset V$ with $I\cap V^* = \emptyset$ satisfying the following:
\begin{itemize}
\item[(i)]{} $\alpha (\cH) \geq |I|
+ \alpha(\cH^*)$ \hskip 2em where $\cH^*=\cH[V^*]$
\item[(ii)]{} $\displaystyle |I| \geq \frac{0.99}{e}  w_s\frac{n}{t}$ \hskip 2em where $w_s := (s+1)^{\frac{1}{k-1}} - s^{\frac{1}{k-1}}$
\item[(iii)]{} $\displaystyle \frac{n}{e}  (1 - \varepsilon) \leq |V^*| \leq \frac{n}{e} $ \hskip 2em where $\varepsilon=1/(10^6\ln T)$
\item[(iv)]{} for every vertex $v\in V^*$ and $i=2, \ldots , k$ it is
  $$d_i^*(v) \leq {k-1 \choose i-1}
 (s+1)^{\frac{k-i}{k-1}}  \(\frac{t}{e}(1 + \varepsilon)\)^{i-1},$$
where $d_i^*(v)$ is the number of
$i$-element edges in $\cH^*$ containing $v$.
\end{itemize}
\end{lemma}

 In order to apply Lemma~\ref{lemma1}, we are going to obtain an induced subhypergraph of $\cH$ in Lemma~\ref{lem.uncrow.new} that satisfies the assumptions of Lemma~\ref{lemma1}. 

\begin{lemma} \label{lemma2}
Let $k \geq 2$ be a fixed integer. Let $T$, $N$ and $s$ be positive integers satisfying 
\begin{equation*}
T\geq T_0(k), \hskip 0.5em N\geq N_0(k,T), \hskip 0.5em s= 0.001 \ln T.
\end{equation*}

Let $\cH =(V, \cE_2 \cup \dots \cup \cE_k)$ 
be a hypergraph with $N$ vertices and  average degrees $t_i^{i-1}:=i|\cE_i|/N$ satisfying that
$$t_i^{i-1} \leq c_i  T^{i-1}  (\ln T)^{\frac{k-i}{k-1}},$$
where the constants $c_i > 0$ are such that $c_i < \frac{1}{{16}k^2}  {k-1\choose i-1}
 10^{-3\frac{k-i}{k-1}}$. 

Then, $\cH$ contains a subhypergraph
$\cH^* =(V^*, \cE_2^* \cup \dots \cup \cE_k^*)$ 
induced on $V^*$ with $|V^*|=n$ such that the following hold.
\begin{enumerate}
\item 
$\displaystyle \frac{3}{4}  \frac{N}{e^{s}}  \leq n \leq   \frac{N}{e^{s}}$,
and
\item for each vertex $v \in V^*$, 
\begin{eqnarray} \label{zoo}
d_i^*(v) \leq {k-1 \choose i-1}  s^{\frac{k-i}{k-1}}  
\(\frac{T} {e^s}\)^{i-1},
\end{eqnarray}
where $d_i^*(v)$ is the number of $i$-element edges in $\cH^*$ containing $v$.
\end{enumerate}
\end{lemma}

\begin{proof}
Let $V'$ be a random set obtained by choosing each vertex in $\cH$ independently with probability $p:= 1/e^{s}$. For its expected size we 
 have 
$\EE[|V'|] = 
Np=N/e^s$,
thus by Chernoff's inequality it is
$$\mbox{Pr} \big(|V'|\leq \EE[|V'|]-u \big) \leq  e^{-u^2/N}
$$ for every real $u \geq 0$.
 Then, with  
$s = 10^{-3}  \ln T$ and a sufficiently large $N\geq N_0(k,T)$ we have \begin{eqnarray} \label{poi0}
\mbox{Pr}\left(  |V'|\leq \EE[|V'|]-\frac{1}{8}\frac{N}{e^s}
 \right) \leq 
  e^{- \frac{N^2/(64  e^{2s})}{
N}} 
=  e^{- N/(64  e^{2s})
} 
< \frac{1}{k}\, .
\end{eqnarray}
Let $\cH' = (V',
\cE_2' \cup \dots \cup 
\cE_k')$ be the
subhypergraph of $\cH$ induced  on $V'$.
For $i=2, \dots , k$,
we have \begin{equation*} \EE[|\cE_i'|] = p^i  |\cE_i|
= p^i   t_i^{i-1} N/i  
\leq   e^{-si} c_i T^{i-1}  (\ln T)^{\frac{k-i}{k-1}}  N/i.
\end{equation*}
By Markov's inequality we infer
\begin{eqnarray} \label{poi3}
\mbox{Pr} \left( |\cE_i'|  >
k  \EE[|\cE_i'|] \right) \leq 1/k \, .
\end{eqnarray}
By~\eqref{poi0} and~\eqref{poi3},   
there exists an induced 
subhypergraph  $\cH' = (V',
\cE_2' \cup \dots \cup 
\cE_k')$ of $\cH$ such that 
\begin{eqnarray}
|V'| &\geq& \frac{7}{8}  \frac{N}{e^s},  \label{poi4} \\
|\cE_i'| &\leq& k  
  c_i e^{-si}  T^{i-1}  (\ln T)^{\frac{k-i}{k-1}}
 N/i \label{poi5a}
\hskip 1em \mbox{for } i=2, \dots ,k.
\end{eqnarray}

From now on, we are going to obtain an induced subhypergraph $\cH^{*}$ of $\cH'$ such that the maximum degree is bounded. From~\eqref{poi5a}, we infer 
$$\EE[d'_i(v)]\leq k  
 c_i e^{-si}  T^{i-1}  (\ln T)^{\frac{k-i}{k-1}}
 N/|V'|.$$
 Let $Y_i$ be the number of  vertices $v \in V'$ of high degree, i.e.,
 such that $$d_i'(v) 
> 7  k\cdot k    c_i e^{-si}  T^{i-1}
 (\ln T)^{\frac{k-i}{k-1}}  N/|V'|.
$$
Since
\begin{eqnarray*}
 Y_i \cdot  7  k^2    c_i e^{-si}  T^{i-1}
 (\ln T)^{\frac{k-i}{k-1}}  N/|V'| 
 \leq  i  |\cE_i'| 
 \leq k   c_i e^{-si}  T^{i-1} 
 (\ln T)^{\frac{k-i}{k-1}}  N,
\end{eqnarray*}
 we infer
$Y_i \leq |V'|/(7 k)$.  
Thus, the total number of vertices of high degree is bounded as
\begin{equation}\label{eq:Y}
\sum_{i=2}^k Y_i \leq  \sum_{i=2}^k \frac{|V'|}{7k}  \leq 
\frac{|V'|}{7}.
\end{equation}
By deleting these 
vertices of high degree from $V'$, we obtain an induced 
subhypergraph  $\cH^{*} = (V^{*},
\cE_2^{*} \cup \dots \cup 
\cE_k^{*})$ of $\cH^{*}$ with
$|V^{*}| \geq (6/7)  |V'|$ 
 such that with~\eqref{poi4} we infer
\begin{eqnarray}
|V^{*}| &\overset{\eqref{eq:Y}}{\geq}& \frac{6}{7}|V'| \overset{\eqref{poi4}}{\geq} \frac{3}{4}  \frac{N}{e^{s}} \label{eq:V(2)} \\
d_i^{*}(v) &{\leq}&
  7  k^2    c_i e^{-si}  T^{i-1} 
 (\ln T)^{\frac{k-i}{k-1}}  N/|V'| \nonumber  \\
 &\overset{\eqref{poi4}}{\leq}&
 8  k^2 
  c_i  \(\frac{T}{e^s}\)^{i-1} 
 (\ln T)^{\frac{k-i}{k-1}}, \nonumber
\end{eqnarray}
where $d_i^{*}(v) $ is the number of $i$-element edges
in
${\cH}^{*}$ containing $v$.
Recalling $s := 10^{-3}   \ln T$ and the upper bound on $c_i$, we have
$$
d_i^{*}(v)  \leq 
{k-1 \choose i-1}  s^{\frac{k-i}{k-1}}  \(\frac{T}{e^s}\)^{i-1},
$$
which proves condition~(b) in~\eqref{zoo}.
We can obtain the condition $|V^{*}|\leq N/e^s$
by possibly deleting some more vertices from $V^*$. This together with~\eqref{eq:V(2)} implies condition~(a).
\end{proof} 

Now we are ready to prove  Lemma~\ref{lem.uncrow.new}. 
\begin{proof}[Proof of Lemma~\ref{lem.uncrow.new}]
We apply  Lemma~\ref{lemma2} with $s= 10^{-3}  \ln T$ 
to $\cH$ on $N$ vertices. Then 
we obtain a subhypergraph
   $\cH^* =(V^*, \cE_2^*
 \cup \dots \cup \cE_k^*)$ induced on $V^*$ with $|V^*|=n$ such that the following hold:
\begin{enumerate}
\item 
$\displaystyle \frac{3}{4}  \frac{N}{e^{s}}  \leq n \leq   \frac{N}{e^{s}}$,
and
\item for each vertex $v \in V^*$, it is
\begin{eqnarray*} 
d_i^*(v) \leq {k-1 \choose i-1}  s^{\frac{k-i}{k-1}}  
\(\frac{T} {e^s}\)^{i-1},
\end{eqnarray*}
where $d_i^*(v)$ is the number of $i$-element edges in $\cH^*$ containing $v$.
\end{enumerate}

Set $
n_s= n$ and $t_s= T/e^s$ and $\cH_s = \cH^* =
(V_s, \cE_{2; s}
 \cup \dots \cup \cE_{k; s})$.
We apply Lemma~\ref{lemma1}
 with  $\varepsilon =  1/( 10^{6}  \ln T)$ 
 to  $\cH_s$, iteratively. Let $n_r$ and $t_r$ be such that
 \begin{equation}\label{eq:n(2)}
\frac{3}{4}\frac{N}{e^r}   (1 - \varepsilon)^{r-s}  \leq n_r \leq
\frac{N}{e^r} 
\end{equation}
and 
 \begin{equation}\label{eq:t(2)}
t_r = \frac{T}{e^r}  (1+\varepsilon)^{r-s}.
\end{equation}
 For each $r = s, \dots , 10^{-2}  \ln T$, we 
obtain 
an independent set $I_r \subseteq V_r$ and  a vertex set $V_{r+1}\subset V_r$ with $I_r\cap V_{r+1}=\emptyset$ satisfying the following:
\begin{itemize}
\item[(i)]{} $\alpha (\cH_r) \geq |I_r|
+ \alpha(\cH_{r+1})$ \hskip 2em where $\cH_{r+1}=\cH_r[V_{r+1}]$
\item[(ii)]{} $\displaystyle |I_r| \geq \frac{0.99}{e}  w_r\frac{n_r}{t_r}$ \hskip 2em where $w_r := (r+1)^{\frac{1}{k-1}} - r^{\frac{1}{k-1}}$
\item[(iii)]{} $\displaystyle \frac{n_r}{e}  (1 - \varepsilon) \leq |V_{r+1}| \leq \frac{n_r}{e}$ \hskip 2em recalling $\varepsilon=1/(10^6\ln T)$
\item[(iv)]{} for every vertex $v\in V_{r+1}$ and $2\leq i\leq k$,  $$d_{i; r+1}(v) \leq {k-1 \choose i-1}
 (s+1)^{\frac{k-i}{k-1}}  (t_{r+1})^{i-1},$$
where $d_{i; r+1}(v)$ is the number of
$i$-element edges in $\cH_{r+1}$ containing $v$.
\end{itemize}

We first check how many times we can iteratively apply Lemma~\ref{lemma1} to $\cH_r$. Observe that we can apply Lemma~\ref{lemma1}
as far as inequalities~\eqref{eq:s} and~\eqref{eq:n,t} are satisfied.
\begin{itemize}
\item  The inequality~\eqref{eq:s} is $s+r\leq 0.01(\ln T)$. 
\item From~\eqref{eq:n(2)}, the inequalities about $n$ in 
	(\ref{eq:n,t}) are satisfied if $\displaystyle \frac{1}{2}\leq \frac{3}{4}(1-\epsilon)^{r-s}$. Using $1-p\geq e^{-2p}$ for $0\leq p\leq 0.5$, one can check that it suffices to have 
$r\leq 10^5 \ln T.$

\item From~\eqref{eq:t(2)}, the inequalities about $t$ in 
	(\ref{eq:n,t}) are satisfied if $\displaystyle (1+\epsilon)^{r-s}\leq 2$. One can check that it suffices to have $r\leq 10^5 \ln T.$
\end{itemize}
Therefore, we can apply Lemma~\ref{lemma1} for $r+s\leq 0.01\ln T$.

Now we estimate the size of an independent set in $\cH$ obtained by the above procedure. Notice that by using $(1+\epsilon)^n \geq 1 +\varepsilon  n$ 
and $1 + \varepsilon \leq e^{\varepsilon}$ and $r \leq 10^{-2}  \ln T$
and $\varepsilon = 10^{-6}/\ln T$ we have 
\begin{eqnarray*}
 \frac{n_r}{t_r} &\geq & \frac{(3/4)  N  
(1-\varepsilon)^{r-s}/e^r}{   
T  (1+\varepsilon)^{r-s}/e^r }
\geq \frac{3}{4} \frac{  N}{ T}  
\frac{(1-\varepsilon)^{r}}{  (1+\varepsilon)^{r}}  \\
 &\geq& \frac{3}{4} \frac{  N}{ T}   \frac{1-\varepsilon  r}{
 e^{\varepsilon  r}} 
\geq \frac{3}{4}\frac{  N}{ T}   \frac{1-10^{-8}}{
 e^{10^{-8}}}  
 \geq 0.74  \frac{N}{T}.
 \end{eqnarray*}
Hence, recalling that  $w_r = (r+1)^{\frac{1}{k-1}} - r^{\frac{1}{k-1}}$, 
we obtain an independent set $I = I_s \cup \dots 
\cup I_{0.01\ln T}$
in $\cH$ with
\begin{eqnarray*}
 \alpha (\cH) &\geq & |I| = \sum_{r=s}^{0.01\ln T} |I_r| 
\geq  \frac{0.99}{e} 0.74  \frac{N}{T}   \sum_{r=s}^{0.01\ln T}
w_r  \\
&\geq & \frac{0.73}{e}  \frac{N}{T}   \sum_{r=s}^{0.01\ln T} 
\((r+1)^{\frac{1}{k-1}} - r^{\frac{1}{k-1}}\) \\ 
&\geq &   \frac{0.73}{e}  \frac{N}{T} (\ln T)^{\frac{1}{k-1}}  \(0.01^{\frac{1}{k-1}}-0.001^{\frac{1}{k-1}}\),
\end{eqnarray*}
which gives the lower bound~\eqref{tu1}
in Lemma~\ref{lem.uncrow.new}. 
\end{proof}

\section{A weak version of Theorem~\ref{theo.main-new}}\label{sec:linear with large}


Now we show a (seemingly) weaker version of Theorem~\ref{theo.main-new} in which $T$ and $N$ are large and the assumptions of the upper bounds on $t^{i-1}_i$ are a bit  different.

\begin{lemma} \label{lem.linear.new}
Let $ k\geq 2$ be a fixed integer.  Let $T$ and $N$ satisfy $T > T_0(k)$ 
and $ N > N_0(k, T)$.
 Let $\cH = (V, \cE_2 \cup \dots \cup \cE_k)$
 be a linear
hypergraph on $N$ vertices such that  there are no $3$-cycles and $4$-cycles consisting of $2$-element edges.  Let the average degrees
$t_i^{i-1} := i  |\cE_i|/ N$ 
 satisfy   for $i= 2, \dots , k$
  \begin{equation}\label{eq:cond} t_i^{i-1}  \leq c_i  
T^{i-1}  (\ln T)^{\frac{k-i}{k-1}},\end{equation}
where $c_i$ are constants satisfying $0<c_i < \frac{1}{2^9 k^6}  {k-1 \choose i-1} 
 10^{-3\frac{k-i}{k-1}}$.
Then, there exists a constant $C_k > 0$ such that
\begin{eqnarray} \label{gen-szem2}
\alpha (\cH) \geq
C_k    \frac{N}{T}  (\ln T)^{\frac{1}{k-1}}  \; .
\end{eqnarray}
\end{lemma}

\begin{proof}
Let $d_i(v)$ denote the number of $i$-element 
edges in $\cH$ containing vertex $v \in V$.  
 We  delete all vertices 
$v \in V$ 
 with $d_i(v) >  k  c_i   T^{i-1}  
(\ln T)^{\frac{k-i}{k-1}}$ for some $i = 2, \dots , k$.
Let $V^*$ be the set of remaining vertices, hence
 we have $|V^*| \geq N/k$. Then, the subhypergraph 
$\cH^* =(V^*,\cE_2^*
\cup \dots \cup \cE_k^*) $ of
$\cH$ induced on $V^*$ satisfies that for each vertex $v \in V^*$ 
\begin{eqnarray} \label{zu21}
d_i^*(v) \leq  k  c_i   T^{i-1}  
(\ln T)^{\frac{k-i}{k-1}}, 
\end{eqnarray}
where $d^*_i(v)$ is the number of $i$-element edges in $\cH^*$ containing
vertex $v$.

Now we  estimate the numbers of $3$-cycles and of $4$-cycles not containing a $3$-cycle in  
$\cH^*$. First we consider the number of $3$-cycles in $\cH^*$.
 Let $C^*(g,h,i)$ denote the number of $3$-cycles $\{E_1, E_2, E_3\}$ 
in $\cH^*$ such that $|E_1|=g$, $|E_2|=h$, and $|E_3|=i$. To estimate  
$C^*(g,h,i)$, we fix a vertex $v \in V^*$ and regard it as a vertex in $E_1\cap E_2$. 
There are at most $k  c_g   T^{g-1}  
(\ln T)^{\frac{k-g}{k-1}}$ 
edges in $\cE_g^*$
containing vertex $v$. Let $E_1$ be one of these edges. Similarly, there are at most
$k   c_h   T^{h-1}  
(\ln T)^{\frac{k-h}{k-1}}$  edges in $\cE_h^*$
containing $v$. Let $E_2$ be one of 
these edges. Moreover, 
we fix one vertex $w \in E_1$ with $w \neq v$ 
and another vertex $x\in E_2$ with $x\neq v$ in at most $(k-1)^2$
ways, and consider $v, x\in E_3$. Since $\cH^*$ is linear, 
there is at most one $i$-element edge $E_3 \in \cE_i^*$
containing both vertices $w$ and $x$. Hence, for each $2\leq g\leq h\leq i\leq k$,
\begin{eqnarray} \label{zu11}
 C^*(g,h,i) &\leq & N \cdot k  c_g   T^{g-1}  
(\ln T)^{\frac{k-g}{k-1}} \cdot k  c_h   T^{h-1}  
(\ln T)^{\frac{k-h}{k-1}}  \cdot
 (k-1)^2   \nonumber \\
&<& N  k^4  c_g  c_h  T^{g+h-2}   (\ln T)^2  .
\end{eqnarray}

Next, we consider the number of $4$-cycles in $\cH^*$. Let $C^*(g,h,i,j)$ be the number of $4$-cycles $\{E_1,E_2,E_3,E_4\}$ in $\cH^*$ such that 
$|E_1|=g$, $|E_2|=h$, $|E_3|=i$, and $|E_4|=j$ and  any three of $\{E_{1}, E_{2}, E_{3}, E_4\}$ do not  form a $3$-cycle.
 With an argument similar to the argument to obtain~\eqref{zu11}, 
one can show that for each $2\leq g\leq h\leq i\leq j\leq k$,
\begin{eqnarray} \label{zu12}
 C^*(g,h,i,j)  
< 3  N   k^6  c_g  c_h  c_i 
   T^{g+h+i-3}   (\ln T)^{3}.
\end{eqnarray}

Note that the factor $3$ in~\eqref{zu12} arises due to the possible arrangements
of the edges of possibly different sizes.

Now we choose each vertex in $V^*$ independently with probability $p =T^{-1+\varepsilon}$ for
 some constant $\varepsilon
> 0$. Let $V^{**} \subseteq V^*$ be the set of chosen vertices,
and let $\cH^{**} = (V^{**}, \cE_3^{**}
\cup \dots \cup  \cE_k^{**})$ be the
subhypergraph of $\cH$ induced on $V^{**}$. 
Note that 
\begin{equation}\label{eq:V^{**}}
\EE\[|V^{**}|\]=p|V^*|\geq \frac{N}{k}T^{-1+\epsilon}.
\end{equation}

Let $C^{**}(g,i,j)$ and $C^{**}(g,h,i,j)$ be the numbers of 
$3$-cycles and $4$-cycles (not containing a $3$-cycle) in $\cH^{**}$, respectively.
Since a $3$-cycle covers exactly $(g+h+i-3)$ vertices in a linear hypergraph, inequality~\eqref{zu11} yields that for $i\geq 3$,
\begin{eqnarray*}
 \EE[C^{**}(g,h,i)] < p^{g+h+i-3}  N 
 k^4  c_g  c_h    T^{g+h-2} 
  (\ln T)^2 
&\leq&  
 k^4  c_g  c_h 
 N  T^{-(i-1)+\varepsilon(g+h+i-3)}  
(\ln T)^{2} \nonumber \\
&< &
 k^4  c_g  c_h 
   N  T^{-2+3k \varepsilon}  
(\ln T)^{2}. 
\end{eqnarray*}
Moreover, since there are no $3$-cycles
 with $3$ edges, each of size $2$, in $\cH$, we have $\EE[C^{**}(2,2,2)]=0$.
Hence, we infer
\begin{eqnarray}
 \sum_{2\leq g\leq h\leq i\leq k} \EE[C^{**}(g,h,i)] 
<    k^7\max_{2 \leq g \leq k} \{ c_g^2\}  \cdot
  N  T^{-2+3k \varepsilon}  
(\ln T)^{2}. \label{eq:3-cycle}
\end{eqnarray}

Similarly, inequality~\eqref{zu12} implies that for $j\geq 3$,
\begin{eqnarray*}
 \EE[C^{**}(g,h,i,j)] &< & p^{g+h+i+j-4}  3  N   k^6  
c_g  c_h 
c_i 
T^{g+h+i-3}   (\ln T)^{3}
 \\
&=&  3  k^6   c_g  c_h 
c_i  N T^{-(j-1)+\varepsilon(g+h+i+j-4)}  
(\ln T)^{3} \\ 
&\leq& 3k^6  \max_{2\leq g\leq k} \{c_g^3\} \cdot
N  T^{-2+4k \varepsilon}  
(\ln T)^{3}.
\end{eqnarray*}
Also since there are no $4$-cycles with $4$ edges, each of size $2$, in $\cH$, we have $\EE[C^{**}(2,2,2,2)]=0$.
Thus, we infer
\begin{eqnarray}\label{eq:4-cycle}
 \sum_{2 \leq g \leq h \leq i \leq j  \leq k} \EE[C^{**}(g,h,i,j)] 
\leq 3 k^{10}  \max_{ 2\leq g \leq k } \{  c_g^3 \} \cdot
N  T^{-2+4k \varepsilon}  
(\ln T)^{3}.
\end{eqnarray}

By~\eqref{zu21}, we have 
for $ i=2, \dots , k$ that 
\begin{eqnarray}\label{eq:E_i}
\EE[|\cE_i^{**}|] &= & p^i |\cE_i^*| \leq  p^i k  c_i   T^{i-1}  
(\ln T)^{\frac{k-i}{k-1}}\frac{|V^*|}{i}
   = k  c_i
   T^{-1+\epsilon i}  (\ln T)^{\frac{k-i}{k-1}} \frac{|V^*|}{i}.
\end{eqnarray}
Chernoff's and Markov's inequalities with~\eqref{eq:V^{**}}--\eqref{eq:E_i}
imply that
there exists a subhypergraph  $\cH^{**} = 
(V^{**}, \cE_2^{**}
\cup \dots \cup  \cE_k^{**})$ induced on $V^{**}$ such that
\begin{eqnarray}
|V^{**}|  &\geq& \frac{1}{2k}NT^{-1+\epsilon} \nonumber \\
|\cE_i^{**}| &\leq&  (k+2)  k  c_i
   T^{-1+\epsilon i}  (\ln T)^{\frac{k-i}{k-1}} \frac{ |V^*|}{i}
 \label{zu22} \\
\sum_{2 \leq g \leq h \leq i \leq k}
C^{**}(g,h,i) &\leq& (k+2)  k^7\max_{2 \leq g \leq k} \{ c_g^2\}  \cdot
  N  T^{-2+3k \varepsilon}  
(\ln T)^{2} \nonumber \\
\sum_{2 \leq g \leq h \leq i \leq j \leq k} C^{**}(g,h,i,j) &\leq& 
(k+2)  3 k^{10}  \max_{ 2\leq g \leq k } \{  c_g^3 \} \cdot
N  T^{-2+4k \varepsilon}  
(\ln T)^{3} \nonumber.
\end{eqnarray}
For $0 < \varepsilon
< 1/(4  k -1)$ and $T > T_0(k)$, 
the number of $3$-cycles and the number of $4$-cycles (not containing a $3$-cycle) 
are much less than $|V^{**}|$.
 
Let $\cH^{***}=(V^{***}, \cE_2^{***}\cup \cdots \cup \cE_k^{***})$ be the 
subhypergraph obtained from $\cH^{**}$ by removing one vertex from each $3$-cycle and $4$-cycle (not containing a $3$-cycle).
With  $\varepsilon := 1/(4  k)$, we have
\begin{equation}\label{eq:V^{***}}|V^{***}| \geq \frac{1}{4k} N  T^{-1+\varepsilon}.\end{equation} 
We infer that the average degrees $(t_i^{***})^{i-1}$ of $\cH^{***}$ satisfy 
 for $i=2,\dots, k$:
\begin{eqnarray*}
(t_i^{***})^{i-1}&=& \frac{i |\cE_i^{***}|}{|V^{***}|}\leq  \frac{i |\cE_i^{**}|}{|V^{***}|}\overset{\eqref{zu22}}{\leq}  (k+2)  k  c_i 
   T^{-1+\epsilon i}  (\ln T)^{\frac{k-i}{k-1}} \frac{ |V^*|}{|V^{***}|} \\
   &\leq & (k+2)  k  c_i 
   T^{-1+\epsilon i}  (\ln T)^{\frac{k-i}{k-1}} \frac{ N}{|V^{***}|} \\
   &\overset{\eqref{eq:V^{***}}}{\leq} & 4k^2(k+2) c_i T^{\varepsilon(i-1)} (\ln T)^{\frac{k-i}{k-1}} \\
   &=& \frac{4k^2(k+2)}{\varepsilon^{\frac{k-i}{k-1}}} c_i (T^{\varepsilon})^{i-1} (\ln (T^{\varepsilon}))^{\frac{k-i}{k-1}}\\
   &\leq &   32 k^4 c_i (T^{\varepsilon})^{i-1} (\ln (T^{\varepsilon}))^{\frac{k-i}{k-1}}.
\end{eqnarray*}
Since the assumption 
$32 k^4  c_i < \frac{1}{{16}k^2}  {k-1\choose i-1}
 10^{-3\frac{k-i}{k-1}}$ of Lemma~\ref{lem.uncrow.new} is satisfied, this implies that 
there exists a constant $C_k > 0$ such that
\begin{eqnarray*}
 \alpha( \cH^{***} ) &\geq & C_k \frac{|V^{***}|}{T^{\varepsilon}} (\ln (T^{\varepsilon}))^{\frac{1}{k-1}} \overset{\eqref{eq:V^{***}}}{\geq} 
C_k  \frac{N  T^{-1+\varepsilon}/(4  k   
)}{
 T^{\varepsilon} }  
(\ln (
 T^{\varepsilon}))^{\frac{1}{k-1}} \\
&\geq& C_k'  \frac{N}{T}
 (\ln T)^{\frac{1}{k-1}}, 
\end{eqnarray*}
which completes our proof of Lemma~\ref{lem.linear.new}.
\end{proof}


\section{Proof of Theorem~\ref{theo.main-new}}\label{sec:linear}

We are going to show Theorem~\ref{theo.main-new} from Lemma~\ref{lem.linear.new} that is a weaker version of Theorem~\ref{theo.main-new}. We need to modify two assumptions of Lemma~\ref{lem.linear.new}: the first assumption is about the upper bounds on $t^{i-1}_i$, and the second assumption is about the ranges of $T$ and $N$.

We first change the assumptions on the upper bounds on $t^{i-1}_i$, and show the following.

\begin{lemma} \label{lem.main-new with c}
Let $ k\geq 2$ be a fixed integer.  Let $T$ and $N$ satisfy $T > T_0(k)$ 
and $ N > N_0(k, T)$.
  Let $\cH = (V, \cE_2 \cup \dots \cup \cE_k)$
 be a linear
hypergraph on $N$ vertices such that  there are no $3$-cycles and $4$-cycles consisting of $2$-element edges.  Let the average degrees
$t_i^{i-1} := i  |\cE_i|/ N$ 
 satisfy that  for $i= 2, \dots , k$
  \begin{equation}\label{eq:cond2}t_i^{i-1}  \leq  
T^{i-1}  (\ln T)^{\frac{k-i}{k-1}}.\end{equation}
Then, there exists a constant $C_k > 0$ such that
\begin{eqnarray*} \label{szem-new with c}
\alpha (\cH) \geq
C_k    \frac{N}{T}  (\ln T)^{\frac{1}{k-1}}  \; .
\end{eqnarray*}
\end{lemma}

\begin{proof} We are going to use Lemma~\ref{lem.linear.new}. To this end, we need to change the assumption~\eqref{eq:cond2} to the shape of the assumption~\eqref{eq:cond} in Lemma~\ref{lem.linear.new}. We have
\begin{eqnarray*}
t^{i-1}_i &\leq& T^{i-1}(\ln T)^{\frac{k-i}{k-1}}= c_i\cdot \frac{2}{c_i}T^{i-1}\cdot \frac{1}{2}(\ln T)^{\frac{k-i}{k-1}} \\
&\leq & c_i \(\frac{T}{c^*_i}\)^{i-1}\(\frac{1}{2}\ln T\)^{\frac{k-i}{k-1}} 
\leq  c_i \(\frac{T}{c^*_i}\)^{i-1}\(\ln \frac{T}{c^*_i}\)^{\frac{k-i}{k-1}},
\end{eqnarray*}
where $c^*_i:=(c_i/2)^{\frac{1}{i-1}}$ and the last inequality holds because $T$ is sufficiently large depending on $k$.

Now we apply Lemma~\ref{lem.linear.new}, and we infer
\begin{eqnarray*}
\alpha(\cH)&\geq& C_k\frac{N}{T/c^*_i}\(\ln \frac{T}{c^*_i}\)^{\frac{1}{k-1}} \geq C_kc^*_i\frac{N}{T}\cdot \frac{1}{2}(\ln T)^{\frac{1}{k-1}} \\
&\geq& C'_k \frac{N}{T}(\ln T)^{\frac{1}{k-1}},
\end{eqnarray*}
where $C'_k:= C_k c^*_i/2$ and 
the second inequality holds since $T$ is sufficiently large depending on $k$.
This completes the proof of Lemma~\ref{lem.main-new with c}.
\end{proof}

In order to show Theorem~\ref{theo.main-new}, it only remains to enlarge the ranges of $T$ and $N$ in Lemma~\ref{lem.main-new with c} as $T\geq 1.5$ and every $N$.

First, we enlarge the range of $T$ from $T> T_0(k)$ to $T\geq 1.5$. If $1.5\leq T\leq T_0(k)$, then $T\ln T\geq 1/2$, and hence, Theorem~\ref{turan.2.non} with $t^{i-1}_i\leq (T\ln T)^{i-1}$ implies that 
\begin{eqnarray*}
\alpha(\cH) \geq \frac{1}{4}\frac{N}{T\ln T} \geq \frac{1}{4}\frac{N}{T\ln T_0}
&=& \(\frac{1}{4 (\ln T_0)^{1+\frac{1}{k-1}}}\)\frac{N}{T}(\ln T_0)^{\frac{1}{k-1}}\\
&\geq& \(\frac{1}{4 (\ln T_0)^{1+\frac{1}{k-1}}}\)\frac{N}{T}(\ln T)^{\frac{1}{k-1}}=C_k\frac{N}{T}(\ln T)^{\frac{1}{k-1}}.
\end{eqnarray*}

Next, we enlarge the range of $N$ from $N>N_0(k, T)$ to be every $N$.
Let 
$\cH = (V,
\cE)$ be a hypergraph satisfying the assumption in Lemma~\ref{lem.main-new with c} except for $T> T_0(k)$ and $N> N_0(k,T)$, and suppose that $ N \leq N_0(k,T)$ and $T\geq 1.5$.

Let $L> N_0(k,T)/N$, and consider the hypergraph $\cH'$ obtained by $L$ 
vertex-disjoint copies of  
$\cH$. Observe that $\cH'$ has $LN$ vertices and its 
average degree is the
same as the average degree of $\cH$. Hence $\cH'$ satisfies the assumption of
Lemma~\ref{lem.main-new with c}, thus
 $$\alpha (\cH') 
\geq C_k  \frac{LN}{T}  (\ln T)^{\frac{1}{k-1}}.$$
Since  $\alpha (\cH)  = \alpha (\cH')  /L$, we infer $$\alpha (\cH) 
\geq C_k \frac{N}{T} (\ln T)^{\frac{1}{k-1}},$$
which completes the proof of Theorem~\ref{theo.main-new}. \qed

\section{Proof of Theorem~\ref{thm:application_a}}\label{sec:proof_application}

 It will be convenient here to work with the $O$-notation.
	 For functions $f,g \colon {\mathbb N}
 \longrightarrow {\mathbb N}$ and a fixed  integer $\ell > 0$,
 let $f= O_{\ell} (g)$ mean that there exists a constant $c > 0$, depending only on $\ell$, such that $f(n) \leq c g(n)$ for every sufficiently large $n \in {\mathbb N}$.

\begin{proof}[Proof of the lower bound of Theorem~\ref{thm:application_a}]  
 We will  use in our arguments some ideas 
from~~\cite{ALR91}
(compare also~~\cite{LRW96}). Recall that $\cH=\cH(n;2,\ldots  ,\ell)$ is the hypergraph on the vertex set $V$, $|V|=n$, in which, for $s=2, \ldots, \ell$, each $s$-subset of $V$ is a hyperedge of $\cH$, and let a $(u_2,  \ldots u_{\ell})$-bounded edge-coloring $\Delta$ of $\cH$ be given.  We define another hypergraph $\cG=(V,E)$ as follows: If $e_1$ and $e_2$ are hyperedges in $\cH$ with the same color in $c$, then $e_1\cup e_2 \in E(\cG)$. Let $E_{2i}$ be the set of hyperedges of $\cG$ of size $2i$, hence $E=
	\bigcup_{i=2}^{\ell} E_{2i}$. Observe that if $I\subset V$ is an independent set of $\cG$, then the subhypergraph $\cH[I]$ of $\cH$ induced on $I$ is totally multicolored. Hence, a lower bound on $f_{\Delta}(n;u_2,\ldots,u_{\ell})$ can be obtained by finding an independent set in $\cG$.
Therefore,  it suffices to show the following:
\begin{enumerate}
\item[(i)]  Under assumption~\eqref{cond1_a}, we have 
 \begin{equation*}\label{eq:result1}\alpha(\cG)\geq c_1\(\frac{n^s}{u_s}\)^{\frac{1}{2s-1}}.
\end{equation*}
\item[(ii)]  Under assumptions~\eqref{cond2_a} and~\eqref{cond3_a}, we have 
	\begin{equation*}\label{eq:result2}
		{ \alpha(\cG)\geq c_3\(\frac{n^s}{u_s}{ \ln n}\)^{\frac{1}{2s-1}}.   }
\end{equation*}
\end{enumerate}

We first prove (i). Let assumption \eqref{cond1_a} hold.
Set $$p=\(\frac{1}{n^{s-1}u_s}\)^{\frac{1}{2s-1}}.$$
Let $R$ be a random subset of $V$ obtained by choosing each vertex independently with probability $p$. Note that with high probability $|R|=np(1+o(1))$.

  For $i=2, \ldots, \ell$, let $E_{2i}^R$ be the set of all $2i$-element hyperedges in the subhypergraph $\cG[R]$ 
induced on $R$.
To estimate the expected numbers $\EE(|E^R_{2i}|)$ of $2i$-element hyperedges,
	choose an $i$-element hyperedge $e$ in $\binom{n}{i}$ ways. Less than 
	$u_i$ other  hyperedges have the same color as $e$, thus
	$$
		|E_{2i}| \leq  \frac{n^iu_i}{2(i!)} \hspace{1cm} \mbox{hence,} 
\hspace{1cm} \EE(|E^R_{2i}|)\leq  \frac{n^iu_ip^{2i}}{2(i!)}.  
	$$
Consequently, Markov's inequality gives  for $i=2,\ldots,\ell$ that
\begin{equation*}
\Pr\[|E^R_{2i}|>  \ell\cdot \frac{n^iu_ip^{2i}}{2(i!)}\]\leq \frac{1}{\ell}.
\end{equation*}

Therefore, there exists a vertex set $R\subset V$ such that
\begin{itemize}
\item $|R|\geq np/2$,
\item $\displaystyle |E^R_{2i}|\leq  \ell\cdot \frac{n^iu_ip^{2i}}{2(i!)}$\hskip 1em for $i=2,\dots,\ell$.
\end{itemize}
The average degree $t_{2i}^{2i-1}$ for the $2i$-element hyperedges in the subhypergraph $\cG[R]$ satisfies
$$ t_{2i}^{2i-1}\leq \(2i\ell\cdot \frac{n^iu_ip^{2i}}{2(i!)}\)\Big/\(np/2\)=\frac{2\ell}{(i-1)!}n^{i-1}u_ip^{2i-1},$$
and hence, for $i=2,\ldots , \ell$, and some constant $c_{\ell} > 0$, it is
$$t_{2i}\leq c_{\ell}\(n^{i-1}u_i\)^{1/(2i-1)}p\stackrel{\eqref{cond1_a}}{\leq} c_{\ell}\(n^{s-1}u_s\)^{1/(2s-1)}p.$$ 
 By Theorem~\ref{turan.2.non} we infer that 
$$\alpha(\cG)\geq \frac{1}{4}\frac{np/2}{c_{\ell}\(n^{s-1}u_s\)^{1/(2s-1)}p}=
 c'_{1} \(\frac{n^s}{u_s}\)^{1/(2s-1)},
$$
for some constant $ c'_{1}  > 0$, depending only on $\ell$, which completes the proof of (i).

Next we show (ii).
Suppose that~\eqref{cond2_a} and~\eqref{cond3_a} hold. Set 
\begin{equation}\label{eq:second p}p=\(\frac{1}{n^{s-1}u_s}\)^{\frac{1}{2s-1}}\omega, \hskip 2em
\mbox{where $\omega:= \(u_s^2/n\)^{\frac{1}{2(2s-1)(2\ell+1)}}.$}
\end{equation} 
 As above, let $R$ be a random subset of $V$ obtained by choosing each vertex independently with probability $p$. 
We again have that
\begin{itemize}
\item $|R|=np(1+o(1))$ with high probability,
\item For $i=2,\ldots ,\ell$, \begin{equation*}
		\Pr\[|E^R_{2i}|>  3\ell \cdot \frac{n^iu_ip^{2i}}{2(i!)}\]\leq { \frac{1} {3\ell} }.
\end{equation*}
\end{itemize}

Next, we consider $2$-cycles in $\cG$. For integers $2\leq i, j, k \leq \ell$, let $C(2i,2j,k)$ be the family of all $2$-cycles which consist of two distinct hyperedges in $\cG$ with one of size $2i$ and the other of size $2j$ sharing exactly $k$ vertices. Note that the condition on $k$ is either $2 \leq k\leq 2i-1$ for $i=j$ or $ 2 \leq k \leq 2i$ for $ i<j$.  We estimate $|C(2i,2j,k)|$ as follows. Fix the first hyperedge $e\in E_{2i}$ as an arbitrary one in at most $n^iu_i$ ways. Let $e_1\cup e_2\in E_{2j}$ ($e_1,e_2\in E_j(\cH)$) be the second hyperedge in $\cG$ and let $k_1=|e_1\cap e|$ and $k_2=|e_2\cap e|$. Without loss of generality, we assume $k_1\geq k_2$.
\begin{itemize}
\item Suppose $k_1=k$ and $k_2=0$. The number of choices of $e_1$ is at most $O_{\ell}(n^{j-k})$. Since the color of $e_1$ is determined, the number of choices of $e_2$ is at most $u_j\leq n$ because of the assumption that each color class is a matching. Hence, the number of choices of $(e_1,e_2)$ is at most $ O_{\ell}(n^{j-k+1})$.
\item Otherwise, we have $k_2\neq 0$. The number of choices of $e_1$ is  at most $O_{\ell}(n^{j-k_1})$. Then, the number of choices of $e_2$ is  at most $O_{\ell}(1)$. Hence, the number of choices of $(e_1,e_2)$ is  at most $O_{\ell}(n^{j-k_1})$. By minimizing $k_1$, we have the upper bound $O_{\ell}\(n^{j- \lceil k/2\rceil }\)$.
\end{itemize}
Consequently, we have that for all $k \geq 2$,
\begin{equation*}
	|C(2i,2j,k)|=  O_{\ell}\(|E_{2i}|\cdot n^{j- \lceil k/2\rceil }\)=O_{\ell}\(u_i n^{i+j-  \lceil k/2\rceil } \).
\end{equation*}

Let $C^R(2i,2j,k)$ be the random set of all $2$-cycles in $C(2i,2j,k)$ that are contained in $R$. Let 
\begin{eqnarray*}
C^R(2i,2i)&=&\bigcup_{k=2}^{2i-1}C^R(2i,2i,k) \hskip 1em \mbox{for $2\leq i\leq \ell$, \hskip 2em and } \\
C^R(2i,2j)&=&\bigcup_{k=2}^{2i}C^R(2i,2j,k) \hskip 1em \mbox{ for $2\leq i<j\leq \ell$.}
\end{eqnarray*}
Since $\EE\[|C^R(2i, 2j, k)|\]=|C(2i,2j,k)|p^{2i+2j-k}$, we infer that for $2\leq i\leq \ell$,
\begin{eqnarray*}
	\EE\[|C^R(2i,2i)|\]&=&\sum_{k=2}^{2i-1} \EE\[|C^R(2i,2i,k)|\] \stackrel{(np\gg 1)}{=}  O_{\ell}\( u_in^{2i-1}p^{4i-2}+u_in^{i+1}p^{2i+2}\) \\ &\stackrel{(np^2\ll 1)}{=} &O_{\ell}\(u_in^{i+1}p^{2i+2}\),
\end{eqnarray*}
and that for $2\leq i < j\leq \ell$,
\begin{eqnarray*}
\EE\[|C^R(2i,2j)|\]&=& \sum_{k=2}^{2i} \EE\[|C^R(2i,2j,k)|\] \stackrel{(np\gg 1)}{=}  O_{\ell}\(u_i n^{j+i-1}p^{2j+2i-2}+u_in^{j}p^{2j}\) 
\\ &\stackrel{(np^2\ll 1)}{=}&O_{\ell}\(u_in^{j}p^{2j}\).
\end{eqnarray*}
Using Markov's inequality, it simultaneously holds with probability bigger than $2/3$ that
\begin{eqnarray*}
|C^R(2i,2i)|&=&  O_{\ell}\(u_in^{i+1}p^{2i+2}\) \hskip 1em
\mbox{ for $2\leq i\leq \ell$, \hskip 2em and} \\
|C^R(2i,2j)|&=&  O_{\ell}\(u_in^{j}p^{2j}\) \hskip 2.8em
\mbox{ for $2\leq i<j\leq \ell$}.
\end{eqnarray*}
Therefore, there exists a subset $R \subset V$ such that for some constant $c>0$, depending only on $\ell$, we have that
\begin{itemize}
	\item $|R| \geq np/2$
	\item $|E_{2i}^R| \leq cu_i n^i p^{2i}$ \hskip 6.25em for $2\leq i\leq \ell$
	\item $|C^R(2i,2i)| \leq cu_i n^{i+1} p^{2i+2}$\hskip 1.7em for $2\leq i\leq \ell$, \hskip 2em and \\ $|C^R(2i,2j)|\leq  cu_in^{j}p^{2j}$ \hskip 3.05em for $2\leq i<j \leq \ell$.
\end{itemize}
  We claim that $\sum_{i=2}^{\ell}|C^R(2i,2i)|\ll np$ and $\sum_{2\leq i<j\leq \ell} |C^R(2i,2j)|\ll np$. For the first inequality, it suffices to check that for $2\leq i\leq \ell$,
 $$ u_in^{i+1}p^{2i+2}\ll np,
\mbox{\hskip 1em that is, \hskip 1em } u_ip(np^2)^{i}\ll 1,$$
indeed,  we have that
\begin{equation}\label{eq:aux}
 { u_ip\overset{\eqref{cond2_a}}{\leq}\(\frac{u_s^2}{n}\)^{\frac{i-1}{2s-1}}\omega  (\ln n)^{\frac{ 2(s-i)}{2s-1}}  }
\mbox{\hskip 1em and \hskip 1em } np^2=\(\frac{n}{u_s^2}\)^{\frac{1}{2s-1}}\omega^2.
\end{equation}
Hence, $$u_ip(np^2)^{i}{ \leq \(\frac{n}{u_s^2}\)^{\frac{1}{2s-1}}\omega^{2i+1}  (\ln n)^{\frac{2(s-i)}{2s-1}} }  \ll 1,$$
where the last inequality follows from  $u_s\geq n^{\frac{1}{2} + \varepsilon}$ and $\omega=\(u_s^2/n\)^{\frac{1}{2(2s-1)(2\ell+1)}}.$
Next, for the second inequality,  similarly it follows that
$u_in^{j}p^{2j} \ll np$ for $2\leq i<j\leq \ell$.


After deleting a vertex from each member of $C^R(2i, 2i)$ or $C^R(2i, 2j)$, there exists a vertex set $U\subset R$ such that
\begin{itemize}
	\item $\displaystyle |U|\geq \frac{np}{ 4}$
\item $\displaystyle |E^{U}_{2i}|\leq   cu_in^ip^{2i}$ \hskip 1em for $2\leq i\leq \ell$.
\item There is no 2-cycle in the subhypergraph $\cG[U]$ induced on $U$.
\end{itemize}

 Set 
\begin{eqnarray} \label{eq:T}
	T &:=& c^*  p (n^{s-1} u_s )^{\frac{1}{2s-1}} 
	(\ln n)^{\frac{2s-2\ell}{(2s-1)(2\ell-1)}}, 
\end{eqnarray}
where  $c^*>0$ is a suitable large constant to be fixed later.
We now check that the average degree $t_{2i}^{2i-1}$ of the subhypergraph 
$\cG[U]$ with hyperedges of size $2i$ satisfies~\eqref{eq:main_condition}, that is, 
\begin{equation}\label{eq:cond_T}
t_{2i}^{2i-1}\leq T^{2i-1} \cdot (\ln T)^{\frac{2\ell - 2i}{2\ell - 1}}.
\end{equation}  
First, observe that
\begin{equation*}
	t_{2i}^{2i-1} \leq 8\ell c u_in^{i-1}p^{2i-1}=8\ell c (u_ip)(np^2)^{i-1}\stackrel{\eqref{eq:aux}}{=}8\ell c \omega^{2i-1}(\ln n)^{\frac{2(s-i)}{2s-1}}.
\end{equation*}
On the other hand, since $u_s\geq n^{1/2+\varepsilon}$, we have 
$\ln T \geq c' \ln n$ for some constant $c'=c'(\ell,\varepsilon) > 0$,
and hence,
\begin{equation*}
T^{2i-1} \cdot (\ln T)^{\frac{2\ell - 2i}{2\ell - 1}}\geq (c^*)^{2i-1}(c')^{\frac{2\ell - 2i}{2\ell - 1}}\omega^{2i-1}(\ln n)^{\frac{2(s-i)}{2s-1}}.
\end{equation*}
 With a suitable large choice of $c^*=c^*(\ell,\epsilon)>0$, depending on $c$ and $c'$, we infer~\eqref{eq:cond_T}.

Theorem~\ref{theo.main-new} implies 
that
there exist  positive constant $c_{\ell}$, depending only on $\ell$, such that
\begin{equation*}
	\alpha(\cG)\geq c_{\ell}\frac{np}{T}(\ln T)^{\frac{1}{2\ell-1}}=
	c_3 \(\frac{n^s}{u_s}{\ln n}\)^{\frac{1}{2s-1}},  
\end{equation*}
 where $c_3 > 0$ is a constant, depending only on $\ell$ and $\varepsilon$, 
which completes the proof of (ii).
\end{proof}

\begin{proof}[Proof of the upper bound  of Theorem~\ref{thm:application_a}]
We will show the following using some ideas 
from~~\cite{ALR91} and~~\cite{Ba85}:
For each integer $k$ with $2\leq k \leq \ell$, 
there exists a positive constant $C=C(k)$
such that for every sufficiently large $n$,
\begin{equation}\label{sak1}
 f(n;u_2,...,u_{\ell})  \leq C  \left(\frac{n^k}{u_k}
  \ln n \right)^{\frac{1}{2k-1}}.
\end{equation}
This with $k=s$ implies the upper bound in Theorem~\ref{thm:application_a}.

For a proof of~\eqref{sak1}, let $\cH_k$ be the subhypergraph of $\cH$ on the vertex set $V$ only with all $k$-element hyperedges. We define a random edge-coloring of $\cH_k$ as follows. Set 
$$m=  c_0  \frac{n^{k}}{u_k},$$ where $c_0$ is a 
 constant with $0\leq c_0\leq 1/(8e^2 (k!) )$. (We ignore divisibility constraints in our arguments, as there is enough room in the calculations.) 
Let $M_{1}, \ldots,M_{m}$ be random matchings 
chosen uniformly and independently from the set of all matchings of size $u_k$
on $V$, and
let $U_0=\emptyset$ and $U_{i}=\bigcup_{j\leq i}M_{j}$ for $i=1,\dots, m$. We color all hyperedges in $M_i 
\setminus U_{i-1}$ by color $i$, and
color the remaining ones with distinct new colors.

In order to prove~\eqref{sak1}, it suffices to show that, for 
$x =C  \left( \frac{n^k}{u_k} 
\ln n \right)^{\frac{1}{2k-1}}$, where $C > 0$ is a suitable
constant to be fixed later, we have
\begin{equation*}\label{eq:upper_goal}
\Pr\Big[ \exists X\subset V \mbox{ such that } |X|=x \mbox{ and } X \mbox{ is totally multicolored } \Big] =o(1).
\end{equation*}
For its proof, let $X \subseteq V$ be an arbitrarily fixed subset of $V$ of size $x$.
We will show that $X$  
is totally multicolored with very small probability. For $i=1,\dots,m$, let $Y_i$ be the number of pairs $\{S,T\}$ of hyperedges in $M_i\setminus U_{i-1}$ contained in $X$.
Observe that $X$ is totally multicolored if and only if 
simultaneously $Y_i=0$ for $i=1,\ldots, m$. 
We can show that  $Y_i=0, \; i=0,\ldots,m,$ hold with very small probability under the condition that the intersection size $|U_{m} \cap [X]^{k}|$ is small, where $[X]^k$ denotes the family of all $k$-element hyperedges in $X$. To this end, let $A$ be the event that $|U_{m} \cap [X]^{k}|\leq c_1 x^k$ where $c_1=1/( 4(k!))$. We have the following:
 \begin{eqnarray} \label{eq4}
\Pr\left[X \mbox{ is totally multicolored }\right] &=& \Pr\left[Y_1=0,\dots, Y_m=0\right] \nonumber \\
&\leq & \Pr[A^c]+\Pr\left[Y_1=0,\dots, Y_m=0, A\right] \nonumber \\
 &\leq & \Pr[A^c]+\Pr\left[Y_1=0,\dots, Y_m=0 \; | \; A\right].
 \end{eqnarray}

  We will use the following two claims:

\begin{claim}\label{clm:first}
For every sufficiently large $n$, \begin{eqnarray} \label{up2} \Pr \left[ A^c\right] \leq \exp\(-c_1  x^{k}\) \; .\end{eqnarray}

\end{claim}

\begin{claim}\label{clm:second} For every  sufficiently large $n$,
\begin{eqnarray} \label{zu1}
\Pr\left[Y_1=0,\dots, Y_m=0 \; | \; A\right] \leq \exp\left(-\frac{c_0c_1^2}{4}  
\frac{u_k x^{2k}}{n^{k}}\right).
\end{eqnarray}
\end{claim}
The union bound and Claims~\ref{clm:first} and \ref{clm:second}  together with~\eqref{eq4} yield that
\begin{eqnarray} \label{up5}
&& \Pr \Big[\exists X \subset V \mbox{ such that } |X|= x  \mbox{ and $X$ totally multicolored} \Big]  \nonumber \\
&\leq& \binom{n}{x} \left( \exp \left(-\frac{c_0c_1^2}{4}    \frac{u_kx^{2k}}{n^{k}}
\right)
+\exp\(-c_1x^{k}\)
\right) 
\nonumber 
\\ 
&\leq& \exp \left( x \ln n \right)\cdot \left( \exp \left(-\frac{c_0c_1^2}{4}   \frac{u_kx^{2k}}{n^{k}}
\right)
+\exp\(-c_1x^{k}\)\right) .
\end{eqnarray}
 By choosing the constant   $C > (4/(c_0 c_1^2))^{\frac{1}{2k-1}} $, the term
 (\ref{up5}) goes to $0$ as $n$ tends to $\infty$, which completes our proof of~\eqref{eq:upper_goal}. It  remains to prove Claims~\ref{clm:first} and ~\ref{clm:second}. 

First, we prove Claim~\ref{clm:first}.
Since
$|U_{m} \cap [X]^{k}| \leq \sum_{i=1}^{m} | M_{i} \cap [X]
^{k}|$ and the events $|M_{i} \cap [X]^{k}| \geq t_{i}$, $i= 1,  \ldots ,
m$, are independent,
 we infer that
\begin{eqnarray*}
 \Pr \left[ | U_{m} \cap \lbrack X\rbrack^{k}  | > t
\right] 
\leq   \Pr \left[ \sum_{i=1}^{m} |M_{i} \cap [X]^{k} | 
\geq t \right] 
\leq  \sum_{(t_{i})_{i=1}^{m} \mbox{ \tiny{s.t.} } \atop t_{i} \geq 0, \sum_{i=1}^{m} t_{i} = t}
\prod_{i=1}^{m} \; \Pr \left[ |M_{i} \cap [X]^{k}|  \geq 
t_{i} \right].
\end{eqnarray*}
Now we estimate $\Pr\left[|M_{i} \cap [X]^{k}| \geq t_i\right]$
for integers $1 \leq i \leq m$ and $t_i\geq 0 $. There are 
$\binom{u_k}{t_i}$ choices for selecting $t_i$ hyperedges in $M_i$. Then the $t_i$ hyperedges are contained in $X$ with probability 
$\binom{x}{kt_i}/\binom{n}{kt_i}$, hence
\begin{eqnarray}\label{up1}
\Pr \left[ |M_{i} \cap [X]^{k}| \geq t_i \right]  \leq  
\binom{u_k}{t_i}\frac{ \binom{x}{kt_i}}{\binom{n}{kt_i}} 
 \leq  \left( \frac{u_k x^{k}}{n^{k}} \right)^{t_i}  . 
\end{eqnarray}
Thus, we infer that
\begin{eqnarray*}
 \Pr \left[ | U_{m} \cap \lbrack X\rbrack^{k}  | > t
\right] 
&\leq &   \sum_{(t_{i})_{i=1}^{m} \mbox{ \tiny{s.t.} } \atop t_{i} \geq 0, \sum_{i=1}^{m} t_{i} = t} \prod_{i=1}^{m} \left( \frac{u_k  x^k
}{n^{k}} \right)^{t_{i}} \nonumber 
\leq \binom{t + m -1}{t}  \left( \frac{u_k  x^{k}}{
n^{k}} \right)^t \\
&\leq & \left( \frac{e(t+m)}{t} \right)^{t} 
\left( \frac{u_k  x^{k}}{ n^{k}} \right)^{t} 
\leq \left( \frac{ e   (t + m)  u_k  x^{k}}{
t  n^{k}} \right)^{t} .
\end{eqnarray*}
Take $t = c_1  x^k$ 
and note that $t =o(m)$. Consequently,
\begin{equation*}
 \Pr \left[ | U_{m} \cap \lbrack X\rbrack^{k}  | > c_1  x^k
\right] \leq \left( \frac{ 2em  u_k  x^{k}}{
c_1  x^k n^{k}} \right)^{c_1  x^k}=\left(\frac{2ec_0}{c_1}\right)^{c_1  x^k}\leq e^{-c_1  x^k},
\end{equation*}
where the last inequality follows from $0\leq c_0\leq 1/(8e^2 (k!) )$ and 
$c_1=1/(4 (k!) )$.
This completes our proof of Claim~\ref{clm:first}.

Next, we prove Claim~\ref{clm:second}. 
First, we have that
\begin{eqnarray} \label {k1}
\Pr \Big[Y_1=0,\dots, Y_m=0 \;\Big|\; A\Big]
= \prod_{i=1}^{m} \Pr \Big[Y_i=0 \; \Big|
\; A, Y_j=0 \mbox{ for } 1\leq j\leq i-1\Big].
\end{eqnarray}
For simplification, let $B_{i}$ denote the event that $A$ happens and $Y_j=0 \mbox{ for } 1\leq j\leq i$. Next we upper bound $\Pr [Y_i=0 \; |
\; B_{i-1}]$, or equivalently, we lower bound 
$\Pr [Y_i\geq 1 \; | \; B_{i-1}]$.

To this end, it is useful to consider $\EE[Y_i \;|\; B_{i-1}]$.
The condition  that $A$ holds implies $|U_{i-1} \cap [X]^{k}| \leq  c_1  x^{k}$, that is, 
$$|\lbrack X\rbrack^k \setminus U_{i-1}| \geq \binom{x}{k} - c_1  x^k 
\geq \(\frac{1}{2(k!)}-c_1\)x^k=c_1x^k.$$
For each hyperedge $S \in \lbrack X\rbrack^k$ there are at most
 $k \binom{x-1}{k-1}$ hyperedges in $X$ which are not disjoint from $S$. Hence, the
number of pairs  $\{ S, T\} \in \lbrack \lbrack X\rbrack^k \setminus U_{i-1} 
\rbrack^2$ of hyperedges with $S \cap T = \emptyset$ 
is,  for  every sufficiently large $n$,  at least 
\begin{eqnarray*} \label{e9}
\frac{1}{2} c_1x^k 
 \left( c_1  x^k - k  \binom{x-1}{k-1}
\right) > \frac{c_1^2x^{2k}}{3}.
\end{eqnarray*}
For disjoint hyperedges $S$ and $T$, 
\begin{eqnarray*} \label{e10}
\Pr \left[ S, T \in M_i \right]
 &=& \frac{u_k  (u_k-1)}{\binom{n}{k} 
\binom{n-k}{k}} 
\geq \frac{u_k^2}{n^{2k}},
\end{eqnarray*}
and therefore 
 \begin{eqnarray} \label{ot1}
\EE [Y_{i} \; | \; B_{i-1}] &>& \frac{c_1^2}{3} \Big(\frac{u_k  x^{k}}{n^{k}}\Big)^2. 
\end {eqnarray}

Now we estimate $\Pr [Y_i\geq 1 \; | \; B_{i-1}]$ by using $\EE[Y_i \;|\; B_{i-1}]$. We have that
\begin{eqnarray}\label{eq:Y_i}
\Pr [Y_i\geq 1 \; | \; B_{i-1}] &=& \EE[Y_i \;|\; B_{i-1}] -\sum_{j\geq 2}(j-1)\Pr [Y_i=j \; | \; B_{i-1}] \nonumber \\
&=&\EE[Y_i \;|\; B_{i-1}] -\sum_{j\geq 2}\Pr [Y_i\geq j \; | \; B_{i-1}].
\end{eqnarray}
Observe that  for $j$ pairwise distinct two-element
sets, the underlying set has cardinality
at least $\lceil \, \sqrt{2j + 1} \, \rceil$. Hence, 
\begin{eqnarray} \label{ps1}
\Pr [Y_{i} \geq j  \; | \;  B_{i-1}] \leq
\Pr\left[ | M_i \cap \lbrack X\rbrack^{k} | \geq
\lceil \,\sqrt{2j + 1} \, \rceil \right] \stackrel{\eqref{up1}}{\leq} 
\left(\frac{u_k  x^{k}}{n^{k}}\right)^{\lceil
\, \sqrt{2j + 1}
\,
\rceil}.
\end{eqnarray}
Consequently, it follows from~\eqref{ot1}--\eqref{ps1} and  $x^k = o\(n^k/u_k\)$ that
\begin{equation*}
\Pr [Y_i\geq 1 \; | \; B_{i-1}]\geq \frac{c_1^2}{4}  \(\frac{u_k  x^{k}}{n^{k}}\)^2,
\end{equation*}
that is,
\begin{equation*}
\Pr [Y_i=0 \; | \; B_{i-1}]\leq 1- \frac{c_1^2}{4} \( \frac{u_k  x^{k}}{n^{k}}\)^2.
\end{equation*}

Therefore,~\eqref{k1} gives that
\begin{eqnarray*}
\Pr \Big[Y_1=0,\dots, Y_m=0 \;\Big|\; A\Big]
& \leq & \left(1 - \frac{c_1^2}{4} \frac{u_k^2  x^{2k}}{n^{2k}}
\right)^{m} 
\leq \exp\left(- \frac{c_1^2}{4}   
\frac{u_k^2  x^{2k}m}{n^{2k}}\right) 
\leq \exp\left(-   \frac{c_0c_1^2}{4} 
\frac{u_k  x^{2k}}{n^{k}}\right),
\end{eqnarray*}
which completes our proof of Claim~\ref{clm:second}.
\end{proof}

\end{document}